\documentclass[12pt,twoside]{amsart}
\usepackage{amsmath, amsthm, amscd, amsfonts, amssymb, graphicx}
\usepackage{enumerate}
\usepackage[colorlinks=true,
            linkcolor=magenta,
            urlcolor=cyan,
            citecolor=cyan]{hyperref}
\usepackage{mathrsfs}
\usepackage{fonttable}
\newtheorem{theorem}{Theorem}[section]
\newtheorem{lemma}{Lemma}[section]
\newtheorem{remark}{Remark}[section]
\newtheorem{corollary}{Corollary}[section]

\numberwithin{equation}{section}

\begin{document}
	
\title[On some numerical radius inequalities for Hilbert space operators] {On some numerical radius inequalities for Hilbert space operators}

\author{Mahdi Ghasvareh $^1$, Mohsen Erfanian Omidvar$^2$}

\address{$^{1,2} $ Department of Mathematics, Mashhad Branch, 
Islamic Azad University, Mashhad, Iran.\\
}

\email{mghasvareh@gmail.com}

\email{math.erfanian@gmail.com}


\begin{abstract}
This article is devoted to studying some new numerical radius inequalities for Hilbert space operators. Our analysis enables us to improve an earlier bound of numerical radius due to Kittaneh. It is shown, among others, that if $A\in \mathbb{B}\left( \mathscr{H} \right)$, then
\[\begin{aligned}
  \frac{1}{8}\left( {{\left\| A+{{A}^{*}} \right\|}^{2}}+{{\left\| A-{{A}^{*}} \right\|}^{2}} \right)\le  \omega ^{2}\left( A \right)  \le \left\| \frac{{{\left| A \right|}^{2}}+{{\left| {{A}^{*}} \right|}^{2}}}{2} \right\|-m\left( {{\left( \frac{\left| A \right|-\left| {{A}^{*}} \right|}{2} \right)}^{2}} \right ). 
\end{aligned}\]
\end{abstract}
\maketitle
\pagestyle{myheadings}
\markboth{\centerline {On some numerical radius inequalities for Hilbert space operators  }}
         {\centerline {M. Ghasvareh, M.E. Omidvar }}

\bigskip
\bigskip

\section{Introduction}
Let $\mathcal{B}\left( \mathcal{H} \right)$ denote the ${{C}^{*}}$ -algebra of all bounded linear operators on a complex Hilbert space $\mathcal{H}$ with inner product $\left\langle \cdot,\cdot \right\rangle $. For $A\in \mathcal{B}\left( \mathcal{H} \right)$, let $\omega \left( A \right)$ and $\left\| A \right\|$ denote the numerical radius and the operator norm of $A$, respectively. 
Recall that $\omega \left( A \right)=\underset{\left\| x \right\|=1}{\mathop{\sup }}\,\left\langle Ax,x \right\rangle $. It is well-known that $\omega \left( \cdot \right)$ defines a norm on $\mathcal{B}\left( \mathcal{H} \right)$, which is equivalent to the operator norm $\left\| \cdot \right\|$. In fact, for every $A\in \mathcal{B}\left( \mathcal{H} \right)$,
\begin{equation}\label{038}
\frac{1}{2}\left\| A \right\|\le \omega \left( A \right)\le \left\| A \right\|.
\end{equation}
Also, it is a basic fact that $\omega \left( \cdot \right)$ defines a norm on $\mathcal{B}\left( \mathcal{H} \right)$ which satisfies the power inequality
	\[\omega \left( {{A}^{n}} \right)\le {{\omega }^{n}}\left( A \right)\] 
for all $n=1,2,\ldots $.

In \cite{05}, Kittaneh gave the following estimate of the numerical radius which refines the second inequality in \eqref{038}: 
For every $A$,
\begin{equation}\label{5}
\omega \left( A \right)\le \frac{1}{2}\left\| \left| A \right|+\left| {{A}^{*}} \right| \right\|.
\end{equation}
The following estimate of the numerical radius has been given in \cite{06}:
\begin{equation}\label{36}
\frac{1}{4}\left\| {{\left| A \right|}^{2}}+{{\left| {{A}^{*}} \right|}^{2}} \right\|\le {{\omega }^{2}}\left( A \right)\le \frac{1}{2}\left\| {{\left| A \right|}^{2}}+{{\left| {{A}^{*}} \right|}^{2}} \right\|.
\end{equation}
The first  inequality in \eqref{36} also refines the first inequality in \eqref{038}. This can be seen by using the fact that for any positive operator $A,B\in \mathcal{B}\left( \mathcal{H} \right)$,
\[\max \left( \left\| A \right\|,\left\| B \right\| \right)\le \left\| A+B \right\|.\]
Actually,
\[\frac{1}{4}{{\left\| A \right\|}^{2}}=\frac{1}{4}\max \left( \left\| {{\left| A \right|}^{2}} \right\|,\left\| {{\left| {{A}^{*}} \right|}^{2}} \right\| \right)\le \frac{1}{4}\left\| {{\left| A \right|}^{2}}+{{\left| {{A}^{*}} \right|}^{2}} \right\|.\]

For other properties of the numerical radius and related inequalities, the reader may consult \cite{ms, 11, 10}.     
In this article, we give several refinements of numerical radius inequalities. Our results mainly improve the inequalities in \cite{06}.

\section{Main Results}

\begin{lemma}\label{1}
Let $A\in \mathcal{B}\left( \mathcal{H} \right)$. Then
\[\frac{1}{2}\left\| A\pm {{A}^{*}} \right\|\le \omega \left( A \right).\]
\end{lemma}
\begin{proof}
Since $A+{{A}^{*}}$ is normal, we have
\[\begin{aligned}
   \left\| A+{{A}^{*}} \right\|&=\omega \left( A+{{A}^{*}} \right) \\ 
 & \le \omega \left( A \right)+\omega \left( {{A}^{*}} \right) \\ 
 & =2\omega \left( A \right).  
\end{aligned}\]
Therefore,
\begin{equation}\label{2}
\frac{1}{2}\left\| A+{{A}^{*}} \right\|\le \omega \left( A \right).
\end{equation}
Now, by replacing $A$ by $iA$ in \eqref{2}, we reach the desired result.
\end{proof}
\begin{theorem}\label{3}
Let $A\in \mathcal{B}\left( \mathcal{H} \right)$. Then
\[\frac{1}{4}\left\| {{\left| A \right|}^{2}}+{{\left| {{A}^{*}} \right|}^{2}} \right\|\le \frac{1}{8}\left( {{\left\| A+{{A}^{*}} \right\|}^{2}}+{{\left\| A-{{A}^{*}} \right\|}^{2}} \right)\le {{\omega }^{2}}\left( A \right).\]
\end{theorem}
\begin{proof}
For any $A,B\in \mathcal{B}\left( \mathcal{H} \right)$, we have the following parallelogramm law
	\[{{\left| A+B \right|}^{2}}+{{\left| A-B \right|}^{2}}=2\left( {{\left| A \right|}^{2}}+{{\left| B \right|}^{2}} \right),\]
equivalently
	\[{{\left| \frac{A+B}{2} \right|}^{2}}+{{\left| \frac{A-B}{2} \right|}^{2}}=\frac{{{\left| A \right|}^{2}}+{{\left| B \right|}^{2}}}{2}.\]
Therefore, by the triangle inequality for the usual operator norm and Lemma \ref{1}, we have
\begin{align}
   \frac{1}{4}\left\| {{\left| A \right|}^{2}}+{{\left| {{A}^{*}} \right|}^{2}} \right\|&=\frac{1}{2}\left\| \frac{{{\left| A \right|}^{2}}+{{\left| {{A}^{*}} \right|}^{2}}}{2} \right\| \nonumber\\ 
 & =\frac{1}{2}\left\| {{\left| \frac{A+{{A}^{*}}}{2} \right|}^{2}}+{{\left| \frac{A-{{A}^{*}}}{2} \right|}^{2}} \right\| \nonumber\\ 
 & \le \frac{1}{2}\left\| {{\left| \frac{A+{{A}^{*}}}{2} \right|}^{2}} \right\|+\frac{1}{2}\left\| {{\left| \frac{A-{{A}^{*}}}{2} \right|}^{2}} \right\| \nonumber\\ 
  & =\frac{1}{2}{{\left\| \frac{A+{{A}^{*}}}{2} \right\|}^{2}}+\frac{1}{2}{{\left\| \frac{A-{{A}^{*}}}{2} \right\|}^{2}} \nonumber\\ 
 & \le {{\omega }^{2}}\left( A \right) \nonumber.  
\end{align}
We remark here that if $T\in \mathcal{B}\left( \mathcal{H} \right)$, and if $f$ is a non-negative increasing function on $\left[ 0,\infty  \right)$, then $\left\| f\left( \left| T \right| \right) \right\|=f\left( \left\| T \right\| \right)$. In particular, $\left\| {{\left| T \right|}^{r}} \right\|={{\left\| T \right\|}^{r}}$ for every $r>0$. This completes the proof of the theorem.
\end{proof}
We present a refinement the first   inequality from \eqref{36}. For see this, we need the following lemma, which can found in \cite{04}.

\begin{lemma}\label{02}
Let $A,B\in \mathbb{B}\left( \mathscr{H} \right)$. Then
\[\left\| A+B \right\|\le \sqrt{{{\left\| A^*A+B^*B \right\|}}+2\omega ( B{^*}A)}.\]
\end{lemma}
By the above lemma, we can improve the first inequality in \eqref{36}.
\begin{theorem}\label{4}
Let $A\in \mathbb{B}\left( \mathscr{H} \right)$. Then
\begin{equation}\label{020}
\frac{1}{4}\left\| {{\left| A \right|}^{2}}+{{\left| {{A}^{*}} \right|}^{2}} \right\|\le \frac{1}{2}\sqrt{2{{\omega }^{4}}\left( A \right)+\frac{1}{8}\omega \left( {{\left( {{A}^{*}}-A \right)}^{2}}{{\left( {{A}^{*}}+A \right)}^{2}} \right)}\le {{\omega }^{2}}\left( A \right).
\end{equation}
\end{theorem}
\begin{proof}
Let $A=B+iC$ be the Cartesian decomposition of $A$. Then $B$ and $C$ are self-adjoint operators. One can easily check that
\begin{equation}\label{019}
\frac{{{\left| A \right|}^{2}}+{{\left| {{A}^{*}} \right|}^{2}}}{4}=\frac{{{B}^{2}}+{{C}^{2}}}{2},
\end{equation}
and 
\begin{equation}\label{011}
{{\left| \left\langle Ax,x \right\rangle  \right|}^{2}}={{\left\langle Bx,x \right\rangle }^{2}}+{{\left\langle Cx,x \right\rangle }^{2}}
\end{equation}
for any unit vector $x\in \mathscr{H}$. Of course, the relation \eqref{011} implies 
	\[{{\left\langle Bx,x \right\rangle }^{2}}\left( \text{resp}\text{. }{{\left\langle Cx,x \right\rangle }^{2}} \right)\le {{\left| \left\langle Ax,x \right\rangle  \right|}^{2}}.\]
Now, by taking supremum over $x\in \mathscr{H}$ with $\left\| x \right\|=1$, we get
\begin{equation}\label{013}
{{\left\| B \right\|}^{2}}\left( \text{resp}\text{. }{{\left\| C \right\|}^{2}} \right)\le {{\omega }^{2}}\left( A \right).
\end{equation}
Whence,
	\[\begin{aligned}
   \frac{1}{4}\left\| {{\left| A \right|}^{2}}+{{\left| {{A}^{*}} \right|}^{2}} \right\|&=\frac{1}{2}\left\| {{B}^{2}}+{{C}^{2}} \right\| \quad \text{(by \eqref{019})}\\ 
 & \le\frac{1}{2}\sqrt{\| B ^{4}+ C^4\|+2\omega \left( {{C}^{2}}{{B}^{2}} \right)} \quad \text{(by Lemma \ref{02})}\\ 
 & \le\frac{1}{2}\sqrt{\| B\| ^{4}+ \|C\|^4+2\omega \left( {{C}^{2}}{{B}^{2}} \right)}\\
 & \le \frac{1}{2}\sqrt{2{{\omega }^{4}}\left( A \right)+2\omega \left( {{C}^{2}}{{B}^{2}} \right)} \quad \text{(by \eqref{013})}\\ 
 & \le \frac{1}{2}\sqrt{2{{\omega }^{4}}\left( A \right)+2\left\| {{C}^{2}}{{B}^{2}} \right\|} \quad \text{(by the second inequality in \eqref{038})}\\ 
 & \le \frac{1}{2}\sqrt{2{{\omega }^{4}}\left( A \right)+2{{\left\| B \right\|}^{2}}{{\left\| C \right\|}^{2}}} \\
 &\qquad \text{(by the submultiplicativity of the usual operator norm)}\\ 
 & \le {{\omega }^{2}}\left( A \right)  \quad \text{(by \eqref{013})}
\end{aligned}\]
i.e.,
	\[\frac{1}{4}\left\| {{\left| A \right|}^{2}}+{{\left| {{A}^{*}} \right|}^{2}} \right\|\le \frac{1}{2}\sqrt{2{{\omega }^{4}}\left( A \right)+2\omega \left( {{C}^{2}}{{B}^{2}} \right)}\le {{\omega }^{2}}\left( A \right).\]
Since
	\[\omega \left( {{C}^{2}}{{B}^{2}} \right)=\frac{1}{16}\omega \left( {{\left( {{A}^{*}}-A \right)}^{2}}{{\left( {{A}^{*}}+A \right)}^{2}} \right)\]
we get the desired result \eqref{020}. 
\end{proof}
\begin{remark}
Notice that if $A$ is a self-adjoint operator, then Theorem \ref{3} implies
	\[\frac{1}{2}{{\left\| A \right\|}^{2}}\le {{\left\| A \right\|}^{2}}\]
while from Theorem \ref{4} we infer that
	\[\frac{1}{2}{{\left\| A \right\|}^{2}}\le \frac{\sqrt{2}}{2}{{\left\| A \right\|}^{2}}\le {{\left\| A \right\|}^{2}}.\]
Hence, in this case, Theorem \ref{4} is better than Theorem \ref{3}.
\end{remark}
The next lemma can be found in \cite{07}.
\begin{lemma}
If $A$ and $B$ are positive operators in
$\mathbb{B}\left( \mathscr{H} \right)$,  Then
\[\left\| {{ A }}-{{ B }} \right\|\le \max\{  {{\left\| A \right\|}},{{\left\| B \right\|}}\}-\min \{ m\left( {{A }} \right),m\left( {{ B }} \right) \},\]
where $m\left( {{ A }} \right)=\inf \left\{ \left\langle  A x,x \right\rangle :\text{ }x\in \mathscr{H},\left\| x \right\|=1 \right\}$.
\end{lemma}
\begin{theorem}\label{6}
Let $A\in \mathbb{B}\left( \mathscr{H} \right)$. Then
\[\omega^2 \left( A \right)\le \frac{1}{2}\left[ \left\| {{\left| A \right|}^{2}}+{{\left| {{A}^{*}} \right|}^{2}} \right\|-m\left( {{\left( \left| A \right|-\left| {{A}^{*}} \right| \right)}^{2}} \right )\right].\]
\end{theorem}
\begin{proof}
We can write
\[\begin{aligned}
  & \left\| {{\left( \frac{\left| A \right|+\left| {{A}^{*}} \right|}{2} \right)}^{2}} \right\| \\ 
 & =\left\| {{\left( \frac{\left| A \right|-\left| {{A}^{*}} \right|}{2} \right)}^{2}}-\frac{{{\left| A \right|}^{2}}+{{\left| {{A}^{*}} \right|}^{2}}}{2} \right\| \\ 
 & \le \max \left( \left\| {{\left( \frac{\left| A \right|-\left| {{A}^{*}} \right|}{2} \right)}^{2}} \right\|,\left\| \frac{{{\left| A \right|}^{2}}+{{\left| {{A}^{*}} \right|}^{2}}}{2} \right\| \right)\\
 &\qquad-\min \left( m\left( {{\left( \frac{\left| A \right|-\left| {{A}^{*}} \right|}{2} \right)}^{2}} \right),m\left( \frac{{{\left| A \right|}^{2}}+{{\left| {{A}^{*}} \right|}^{2}}}{2} \right) \right) \\ 
 & \le \left\| \frac{{{\left| A \right|}^{2}}+{{\left| {{A}^{*}} \right|}^{2}}}{2} \right\|-  m\left( {{\left( \frac{\left| A \right|-\left| {{A}^{*}} \right|}{2} \right)}^{2}} \right ). 
\end{aligned}\]
On the other hand, since
\[\omega \left( A \right)\le \frac{1}{2}\left\| \left| A \right|+\left| {{A}^{*}} \right| \right\|,\]
we have
\[{{\omega }^{2}}\left( A \right)\le {{\left\| \left( \frac{\left| A \right|+\left| {{A}^{*}} \right|}{2} \right) \right\|}^{2}}=\left\| {{\left( \frac{\left| A \right|+\left| {{A}^{*}} \right|}{2} \right)}^{2}} \right\|.\]
Consequently,
\[{{\omega }^{2}}\left( A \right)\le \left\| \frac{{{\left| A \right|}^{2}}+{{\left| {{A}^{*}} \right|}^{2}}}{2} \right\|-m\left( {{\left( \frac{\left| A \right|-\left| {{A}^{*}} \right|}{2} \right)}^{2}} \right ) ,\]
as desired.
\end{proof}

Using some ideas of \cite{1}, we prove our last result.
\begin{theorem}
Let $A\in \mathbb{B}\left( \mathscr{H} \right)$ and let $f$ be a continuous function on the interval $\left[ 0,\infty  \right)$ and let $g$ be increasing and concave on $\left[ 0,\infty  \right)$, such that $gof$ is increasing and convex on $\left[ 0,\infty  \right)$. Then
\[f\left( \omega \left( A \right) \right)\le \left\| {{g}^{-1}}\left( \frac{gof\left( \left| A \right| \right)+gof\left( \left| {{A}^{*}} \right| \right)}{2} \right) \right\|\le \frac{1}{2}\left\| f\left( \left| A \right| \right)+f\left( \left| {{A}^{*}} \right| \right) \right\|.\]
\end{theorem}
\begin{proof}
As mentioned above, $gof$ is increasing and convex on $\left[ 0,\infty  \right)$, therefore, from the inequality \eqref{5},
\[\begin{aligned}
   gof\left( \omega \left( A \right) \right)&\le gof\left( \left\| \frac{\left| A \right|+\left| {{A}^{*}} \right|}{2} \right\| \right) \\ 
 & =\left\| gof\left( \frac{\left| A \right|+\left| {{A}^{*}} \right|}{2} \right) \right\| \\ 
 & \le \left\| \frac{gof\left( \left| A \right| \right)+gof\left( \left| {{A}^{*}} \right| \right)}{2} \right\|.  
\end{aligned}\]
Therefore,
\[gof\left( \omega \left( A \right) \right)\le \left\| \frac{gof\left( \left| A \right| \right)+gof\left( \left| {{A}^{*}} \right| \right)}{2} \right\|.\]
Now, since ${{g}^{-1}}$ is increasing and convex, we then have
\[\begin{aligned}
   f\left( \omega \left( A \right) \right)&={{g}^{-1}}\left( gof\left( \omega \left( A \right) \right) \right) \\ 
 & \le {{g}^{-1}}\left( \left\| \frac{gof\left( \left| A \right| \right)+gof\left( \left| {{A}^{*}} \right| \right)}{2} \right\| \right) \\ 
 & =\left\| {{g}^{-1}}\left( \frac{gof\left( \left| A \right| \right)+gof\left( \left| {{A}^{*}} \right| \right)}{2} \right) \right\| \\ 
 & \le \left\| \frac{f\left( \left| A \right| \right)+f\left( \left| {{A}^{*}} \right| \right)}{2} \right\| \\ 
 & =\frac{1}{2}\left\| f\left( \left| A \right| \right)+f\left( \left| {{A}^{*}} \right| \right) \right\|.  
\end{aligned}\]
Thus,
\[f\left( \omega \left( A \right) \right)\le \left\| {{g}^{-1}}\left( \frac{gof\left( \left| A \right| \right)+gof\left( \left| {{A}^{*}} \right| \right)}{2} \right) \right\|\le \frac{1}{2}\left\| f\left( \left| A \right| \right)+f\left( \left| {{A}^{*}} \right| \right) \right\|.\]
\end{proof}

\begin{corollary}
Let $A\in \mathbb{B}\left( \mathscr{H} \right)$. Then for any $r \ge 2$,
\[\begin{aligned}
  & {{\omega }^{r}}\left( A \right) \\ 
 & \le \frac{1}{2}\left\| {{\left| A \right|}^{r}}+{{\left| {{A}^{*}} \right|}^{r}}+{{\left| A \right|}^{\frac{r}{2}}}+{{\left| {{A}^{*}} \right|}^{\frac{r}{2}}}+I-\sqrt{2\left( {{\left| A \right|}^{r}}+{{\left| {{A}^{*}} \right|}^{r}}+{{\left| A \right|}^{\frac{r}{2}}}+{{\left| {{A}^{*}} \right|}^{\frac{r}{2}}} \right)+I} \right\| \\ 
 & \le \frac{1}{2}\left\| {{\left| A \right|}^{r}}+{{\left| {{A}^{*}} \right|}^{r}} \right\|. \\ 
\end{aligned}\]

\end{corollary}
\begin{proof}
Define
\[g\left( x \right)=x+\sqrt{x}\quad\text{  }\!\!\And\!\!\text{  }\quad f\left( x \right)={{x}^{r}},\text{ }\left( r\ge 2 \right)\]
on $\left[ 0,\infty  \right)$. Thus,
\[gof\left( x \right)={{x}^{r}}+{{x}^{\frac{r}{2}}}.\]
One can quickly check that $f$, $g$, and $gof$ satisfy all the assumptions in Theorem \ref{6}. Since
\[{{g}^{-1}}\left( x \right)=\frac{2x+1-\sqrt{4x+1}}{2},\]
we get the desired result.
\end{proof}

\end{document}